\documentclass{elsarticle}
\usepackage{array,amsmath,amstext,amssymb,amsbsy,graphicx,a4wide}
\usepackage{pdfsync}
 \usepackage{color}
\usepackage{txfonts}

\newcommand{\tn}{|||}
\newcommand{\bfs}{\boldsymbol{s}}
\newcommand{\bfE}{\boldsymbol{E}}

\newcommand{\tb}{\tilde{b}}

\newcommand{\bfa}{{\boldsymbol a}}
\newcommand{\bfb}{{\boldsymbol b}}
\newcommand{\bfv}{{\boldsymbol v}}
\newcommand{\bfc}{{\boldsymbol c}}

\newcommand{\bfe}{{\boldsymbol e}}

\newcommand{\bfzero}{{\boldsymbol 0}}

\newcommand{\mcT}{\mathcal{T}}
\newcommand{\mcF}{\mathcal{F}}
\newcommand{\mcE}{\mathcal{E}}

\newcommand{\bfu}{{\boldsymbol u}}
\newcommand{\IR}{{\mathbb{R}}}

\newcommand{\bfw}{{\boldsymbol w}}
\newcommand{\bfn}{{\boldsymbol n}}

\newcommand{\bft}{{\boldsymbol t}}

\newcommand{\bfeps}{{\boldsymbol\varepsilon}}
\newcommand{\bff}{{\boldsymbol f}}

\newcommand{\bfx}{{\boldsymbol x}}

\newcommand{\bbR}{{\mathbb{R}}}

\newtheorem{lem}{Lemma}[section]

\newtheorem{thm}{Theorem}[section]
\newtheorem{rem}{Remark}[section]

\begin{document}

\title{A Simple Nonconforming Tetrahedral Element for
the Stokes Equations}

\author{Peter Hansbo} 
\address{Department of Mechanical Engineering, J\"onk\"oping University, 
SE-55111 J\"onk\"oping, Sweden.}
\author{Mats G. Larson} 
\address{Department of Mathematics and Mathematical Statistics,  Ume{\aa} Univerity,  SE-90187 Ume{\aa},  Sweden.}
\date{}

\begin{abstract}
In this paper we apply a nonconforming rotated bilinear tetrahedral element 
to the Stokes problem in $\bbR^3$. We show that the element is stable in combination with a piecewise linear, continuous, 
approximation of the pressure.  This gives an approximation similar to the well known continuous $P^2$--$P^1$ Taylor--Hood element,
but with fewer degrees of freedom. The element is a stable non--conforming low order element which fulfils Korn's inequality, leading to stability also in the case where the 
Stokes equations are written on stress form for use in the case of free surface flow.
\end{abstract}
\maketitle


\section{INTRODUCTION}

The nonconforming rotated $Q_1$ tetrahedron is derived from the nonconforming 
hexahedral element proposed by Rannacher and Turek \cite{RaTu92} and was applied to linear and nonlinear elasticity problems in \cite{Ha11,Ha12,HaLa16}.
It has properties similar to the hexahedral element, with improved bending behaviour in elasticity, compared to the $P^1$ tetrahedron, and it allows for diagonal mass lumping for explicit time--stepping in dynamic problems.

In this paper, we investigate its properties as a Stokes element, and show that in combination with linear, continuous pressures, it 
is {\em inf--sup}\/ stable. In a sense it is thus a reduced Taylor--Hood \cite{TaHo73} element with fewer degrees of freedom. It is one of the lowest order elements that is stable for Stokes, and, unlike some other low order non--conforming elements \cite{CrRa73,RaTu92}, it fulfills Korn's inequality and can thus handle the strain form of Stokes, cf. \cite{Ha11}.

An outline of the paper is as follows: in Section \ref{q1app} we recall the rotated $Q_1$ element; in Section 3 we apply it to the Stokes equations and prove stability and convergence;
in Section \ref{numex} we give some numerical examples to show the properties of the approximation.

\section{THE ROTATED $Q_1$ APPROXIMATION}\label{q1app}

In order to define a low order approximation which is stable for the Stokes problem, Rannacher and Turek \cite{RaTu92} constructed 
a hexahedral element with nodes on the faces.  Two different kinds of continuity can now be imposed: point-wise continuity and 
average continuity and we will in this paper consider point-wise continuity.  The Rannacher-Turek element may be viewed as an 
extension of the classical Crouzeix-Raviart nonconforming tetrahedral element \cite{CrRa73},  for which point-wise and average 
continuity is identical.  

The difficulty in the construction of the Rannacher-Turek element compared to the Crouzeix-Raviart element 
is that we have six degrees of freedom while the dimension of the space of linears is four,  trilinears eight,  and quadratics ten. Therefore 
we shall start with linears and add two suitable quadratic functions in such a way that the nodal mapping is invertible. To that end let 
$\{\bfe_i\}_{i=1}^3$ be an orthonormal coordinate system 
in $\IR^3$ and let $\hat{K} = [-1,1]^3$ be the reference cube.  The midpoints of the faces of $\hat{K}$ takes the form $\{\bfx_i\}_{i=1}^6 = \{\pm \bfe_i\}_{i=1}^3$,  and therefore has one coordinate equal to $\pm 1$ and the other two equal to $0$.  See Fig.  \ref{fig:refelement} 
for the enumeration of the midpoints.  We now seek a space ${V}(\hat{K})$  of shape functions on $\hat{K}$ such that ${P}_1(\hat{K}) \subset {V}(\hat{K})\subset P_2(\hat{K})$ and that has the values at the six midpoints $\{x_i\}_{i=1}^6$ of the faces as degrees of freedom.  The linear functions  ${P}_1(\hat{K}) = \text{span}(1,x_1, x_2, x_3)$ is a four dimensional vector space and therefore we need to add two quadratic polynomials to obtain a six dimensional space.  To find these quadratic polynomials let  $N: V(\hat{K}) \ni v \mapsto [v(\bfx_i)]_{i=1}^6 \in \IR^6$ be the nodal mapping and note that $N(x_i x_j) = \bfzero$,  for  $i \neq j$,  
and therefore we are restricted to adding polynomials in $\text{span}(x_1^2, x_2^2,x_3^2)$.  Now $N(x_1^2 + x_2^2 + x_3^2) = N(1)$ and 
it is therefore natural to consider the two dimensional space
\begin{equation}
R_2(\hat{K})  = \{ v \in  \text{span}(x_1^2, x_2^2,x_3^2): N(v) \cdot N(1) = 0\} = \text{span}(x_1^2-x_2^2, x_2^2 -x_3^2)
\end{equation} 
where the two (non  unique) basis functions on the right hand side is easily chosen by observing that 
$N(b_1 x_1^2 + b_2 x_2^2 + b_3 x_3^2)\cdot N(1) = 2(b_1 + b_2 + b_3)$.  We define 
\begin{equation}
\boxed{V(\hat{K}) =  P_1(\hat{K}) + R_2(\hat{K}) =  \text{span}(1, x_1, x_2, x_3, x_1^2-x_2^2, x_2^2 -x_3^2)}
\end{equation}
and verify by explicit calculation that the coordinate mapping $N:V(\hat{K})  \rightarrow \IR^6$ is indeed invertible.   Solving for Lagrange basis functions $\varphi_i$ such that $N(\varphi_i) =[ \delta_{ij}]_{j=1}^6$,  gives  
 \begin{equation}
V(\hat{K}) = \text{span}(\varphi_i)_{i=1}^6
 \end{equation}
 where 
 \begin{equation}
\begin{array}{>{\displaystyle}l}
\varphi_1=\varphi(x_1,x_2,x_3)\\[3mm]
\varphi_2=\varphi(-x_1,x_2,x_3)\\[3mm]
\varphi_3=\varphi(x_2,x_1,x_3)\\[3mm]
\varphi_4=\varphi(-x_2,x_1,x_3)\\[3mm]
\varphi_5=\varphi(x_3,x_1,x_2)\\[3mm]
\varphi_6=\varphi(-x_3,x_1,x_2)
\end{array}
\end{equation}
and 
\begin{equation}
\varphi(x_1,x_2,x_3) = \frac16(1 + 3x_1 + 2 x_1^2 - x_2^2 - x_3^2) 
\end{equation}
The terminology rotated $Q_1$ elements is motivated by the fact that $x_1^2 - x_2^2 = (x_1 -x_2)(x_1 + x_2) = \xi_1 \xi_2$ where $\xi_1 = x_1 - x_2$ and $\xi_1+\xi_2$ are the degrees of freedom in a coordinate system rotated $\pi/2$ around the $\bfe_3$ axis.
%
%
%
%
%

In \cite{Ha11}, we made the observation that there is a reference tetrahedron $\hat{T}$
inscribed in the reference hexahedron $\hat{K}$,  with edges that are diagonals of the faces of $\hat{K}$, and thus 
the midpoints of the faces of $\hat{K}$ are precisely the midpoints of the edges of $\hat{T}$,  see Fig.  \ref{fig:refelement}. 
We note that for the reference element all edges have the same length and the centre of gravity is the origin. Using an 
affine map  $F : \hat{T} \rightarrow T$,  we map mid side nodes in the reference configuration to mid edge nodes in the 
physical configuration.
 
Let $\mathcal{T}_h:=\{T\}$ be a conforming, shape regular tetrahedrization of {$\Omega\subset\bbR^3$} with mesh parameter 
$h \in (0,h_0]$. We also let $\mcF_h$ be the set of faces and $\mcE_h$ the set of edges.  We make the standard assumption 
for the Taylor--Hood approximation \cite{BrFo91}, that
\begin{equation}\label{eq:assump}
\text{every $T\in\mathcal{T}_h$ has at least three internal edges}
\end{equation}
We define the non-conforming finite element space 
\begin{equation}\label{spacevA}
\boxed{ V_h := \{ v\in L^2(\Omega) :  v\vert_T \circ F\in V(\hat{T}) \; \;T \in \mcT_h,  \text{$v$ is continuous in the midpoints of $E \in \mcE_{h,I}$}\} 
}
\end{equation}
with midpoint continuity for all interior edges.  We note that $\text{dim}(V_h) = |\mcE_h|$,  the number of edges in $\mcT_h$.

\section{APPLICATION TO THE STOKES EQUATIONS}\label{secelast}

\subsection{Problem Formulation and Finite Element Approximation}

We consider the Stokes equations in a domain 
$\Omega$ in $\bbR^3$: find the velocity 
$\bfu = \left[u_i\right]_{i=1}^3$ and the pressure $p$ such that
\begin{equation} \label{diffelasti}
\left\{
\begin{array}{rcl}
 -\Delta\bfu +\nabla p & = &\bff \quad \mbox{in $\Omega$}\\
\nabla\cdot\bfu & = &\textbf{0}\quad \mbox{in $\Omega$}\\
\bfu& = &\textbf{0}\quad\mbox{on $\partial\Omega$}
\end{array}\right.
\end{equation}
Let us define the spaces
\begin{equation}
W := \{ \bfv: \bfv\in [H^1(\Omega)]^3,\; \bfv \; \text{is zero on $\partial\Omega$}\}
\end{equation}
and 
\begin{equation}
Q := \{ q: q\in L^2(\Omega),\; (q,1)_\Omega =0\}
\end{equation}
where $(\cdot,\cdot)_\Omega$ is the standard $L^2$ scalar product.  Then we have the weak form of (\ref{diffelasti}): 
find $(\bfu,p)\in W\times Q$ such that
\begin{equation}\label{weakelast}
\boxed{ a(\bfu,\bfv)-b(\bfv,p)+b(\bfu,q)=(\bff,\bfv)\quad\forall (\bfv,q)\in W\times Q }
\end{equation}
where the forms are given by
\begin{equation}
 a(\bfv,\bfw):= (\nabla \bfv,\nabla \bfw)_\Omega ,\quad b(\bfv,q):= (\nabla \cdot\bfv,q)_\Omega
\end{equation}

To define the finite element method, we introduce the non-conforming finite element space constructed from the space 
$V_h$ in (\ref{spacevA}) by defining
\begin{equation}
W_h := \{ \bfv: \bfv \in [V_h]^3,\; \bfv \; \text{is zero in the midpoints of $E \in \mathcal{E}_h$ on $\partial\Omega$}\}
\end{equation}
and the space of continuous piecewise linear polynomials 
\begin{equation}
Q_h := \{ q: q \in C^0(\Omega)\cap Q,\; q \; \text{is linear on $T$, $\forall T\in\mathcal{T}_h$}\}
\end{equation}
The finite element method is to find $(\bfu_h,p_h)\in W_h\times Q_h$ such that
\begin{align}\label{eq:method}
A_h(\bfu_h,p_h),  (\bfv,q)) = (\bff, \bfv)_\Omega \qquad \forall (\bfv,q) \in W_h \times Q_h
\end{align}
Here the form is defined by 
\begin{equation}
A_h((\bfv,q), (\bfw,r)) = a_h(\bfv,\bfw) -b(\bfv,r)+b(\bfw,p)
\end{equation}
with 
\begin{equation}
a_h(\bfv,\bfw):=\sum_{T\in\mathcal{T}_h}(\nabla \bfv,\nabla \bfw)_T ,\quad \quad b(\bfv,q):=\sum_{T\in\mathcal{T}_h} (\nabla \cdot\bfv,q)_T 
- (\bfn \cdot \bfv,  q)_{\partial T} = \sum_{T\in\mathcal{T}_h} (\bfv, \nabla q)_T 
\end{equation}
and $\nabla \bfv = \bfv \otimes \nabla$ is the tensor with elements  $(\bfv \otimes \nabla)_{ij} = \partial_j v_i$.
\begin{rem} We will also consider the alternative form 
\begin{align}\label{eq:btilde}
\tb(\bfv,q) = \sum_{T \in \mcT_h} (\nabla \cdot \bfv , q )_T =  \sum_{T \in \mcT_h}  - (\bfv , \nabla q)_T + (\bfn \cdot \bfv,  q)_{\partial T}
\end{align}
This form is preferable  in the presence of natural boundary conditions, and requires less numerical computations,  but we will see that the proof of the inf-sup condition 
is more complicated due to the presence of the trace term on $\partial T$ in the right hand side of (\ref{eq:btilde}).  Throughout the 
paper we will focus our presentation on the form $b$ and we will for each result add a remark on the modifications necessary to obtain 
the corresponding result for $\tb$.  Finally,  we will present an inf-sup result for $\tb$  in \ref{sec:appendix}. 
\end{rem}
\begin{rem} Unlike some nonconforming approximations, the rotated $Q_1$ approximation fulfills Korn's inequality \cite{Ha11}, which
means that we may also use the strain form of Stokes: find the velocity 
$\bfu$ and the pressure $p$ such that
\begin{equation} \label{diffelastistrain}
\left\{
\begin{array}{rcl}
 -2{\boldsymbol\nabla}\cdot\bfeps(\bfu) +\nabla p & = &\bff \quad \mbox{in $\Omega$}\\
\nabla\cdot\bfu & = &\text{\bf 0}\quad \mbox{in $\Omega$}\\
\bfu& = &\text{\bf 0}\quad\text{on $\partial\Omega$}
\end{array}\right.
\end{equation}
where
\begin{equation}
(\bfeps)_{ij} := \frac12\left(\frac{\partial u_i}{\partial x_j}+\frac{\partial u_j}{\partial x_i}\right)
\end{equation}
is the strain tensor and ${\boldsymbol\nabla}\cdot$ denotes matrix divergence. This is of interest in free surface flows where we need zero stress as a natural boundary condition, cf. Section \ref{sec:numex2}.
\end{rem}


\subsection{Norms and Continuity of the Forms}

Define the norms
\begin{equation}
\| \bfv \|_{a_h} := a_h(\bfv,\bfv)^{1/2}, \quad \| q \|_\Omega := \| q \|_{L^2(\Omega)}
\end{equation}
and 
\begin{align}
\tn (\bfv,q) \tn_h^2 = \| \bfv \|^2_{a_h} + \| q \|^2_\Omega + h^2 \| \nabla q \|^2_\Omega
\end{align}
We let  $a \lesssim b $ denote $a \leq C b$ with $C$ a positive constant independent of the mesh parameter.  Then we have 
the following continuities of the forms $b$ and $A_h$.

\begin{lem} There are constants such that for all functions in $(W + W_h) \times Q$,
\begin{align}\label{eq:contb}
\boxed{ b(\bfv,q) \lesssim \| \nabla \bfv \|_{\Omega} \Big( \| q \|_\Omega + h \| \nabla q \|_\Omega \Big) }
\end{align}
and 
\begin{align}\label{eq:contAh}
\boxed{ A_h((\bfv,q),(\bfw,r)) \lesssim \tn (\bfv,q) \tn_h \tn (\bfw,r) \tn_h }
\end{align}
\end{lem}
\begin{proof} Define $[ \bfn \cdot \bfv ] = \bfn_1 \cdot \bfv_1 +  \bfn_2 \cdot \bfv_2$ for an interior face shared by elements $T_1$ and $T_2$ and $[ \bfn \cdot \bfv ] = \bfn \cdot \bfv$ for a face at the boundary belonging to element $T$.  Noting that for a face $F$ we then have 
$ [ \bfn \cdot \bfv(\bfx_E) ] =0$ for the midpoint $\bfx_E$ of the edge $E$.  Therefore $([\bfn \cdot \bfv], 1)_F = 0$,  since the quadrature
 formula based on the midpoints of the edges on a triangle is exact for quadratic polynomials.  We therefore have 
$([\bfn \cdot \bfv], q)_F = ([\bfn \cdot \bfv], (I - P_{0,F}) q)_F$, where $P_{0,F}$ is the $L^2$ projection on constants on the face $F$. With these 
preparations at hand we obtain the following bound
\begin{align}
\sum_{T \in \mcT_h}  (\nabla \cdot\bfv,q)_T 
- (\bfn \cdot \bfv,  q)_{\partial T} 
&=
\sum_{T \in \mcT_h}  (\nabla \cdot\bfv,q)_T 
- \sum_{F\in \mcF_h} ((I-P_{0,F}) [\bfn \cdot \bfv],  (I-P_{0,F}) q)_{F} 
\\
&\leq 
\sum_{T \in \mcT_h}  \|\nabla \cdot\bfv\|_T \|q\|_T 
+  \sum_{F\in \mcF_h} \| (I-P_{0,F}) [\bfn \cdot \bfv]\|_F \|( I - P_{0,F})  q\|_{F} 
\\
&\leq 
\sum_{T \in \mcT_h}  \|\nabla \cdot\bfv\|_T \|q\|_T 
+  \sum_{F\in \mcF_h} h \| \nabla_F [\bfn \cdot \bfv]\|_F \|\nabla_F q\|_{F} 
\\
&\leq 
\sum_{T \in \mcT_h}  \|\nabla \cdot\bfv\|_T \|q\|_T 
+  \| \nabla \bfv \|_T h \|\nabla q\|_T 
\end{align}
which proves (\ref{eq:contb}). Finally,  (\ref{eq:contAh}) follows directly from  (\ref{eq:contb}) and the Cauchy-Schwarz inequality.
\end{proof}

\begin{rem} For $\tb$ defined in (\ref{eq:btilde}) we directly have 
\begin{align}
\tb(\bfv,  q) \leq \| \nabla \cdot \bfv \|_\Omega \| q \|_\Omega
\end{align}
since there are no trace terms on the boundary of the elements.
\end{rem}

%
%
%
%

\subsection{Interpolation}

We shall now define interpolants for the finite element space.   Starting with the pressure space we let 
 $\pi_{h,p}:L^2(\Omega) \rightarrow Q_h$ be a Clement interpolant.  We then have the standard estimate 
\begin{equation}\label{eq:interpol-pressure}
\|q -\pi_{h,p} q\|_{H^m(\Omega)} \lesssim h^{k-m} \| q\|_{H^k(\Omega)}, \qquad 0\leq m \leq k \leq 2
\end{equation}
To construct an interpolant for the velocity space we use component-wise  Scott-Zhang interpolation to 
satisfy the Dirichlet boundary conditions (in the nodes), 
\begin{equation}
\pi_{h,u} \bfv = [\pi_{h,SZ} v_i]_{i=1}^3 
\end{equation}
Here we also have the interpolation estimate 
\begin{equation}\label{eq:interpol-velocity}
\|\bfv -\pi_{h,u} \bfv \|_{H^m(\Omega)} \lesssim h^{k-m} \| \bfv \|_{H^k(\Omega)}, \qquad 0\leq m \leq k \leq 2
\end{equation}
From here on we simplify the notation and write $\pi_{h,p} = \pi_h$ and $\pi_{h,u} = \pi_h$ and interpret the operator
 in the correct way depending on in which space the argument reside.  Combining the estimates we get 
\begin{align}\label{eq:interpol}
\tn (\bfv,q) - (\pi_h \bfv,  \pi_h q) \tn_h \lesssim h( \| \bfv \|_{H^2(\Omega)} + \| q \|_{H^1(\Omega)} ) 
\end{align}

\subsection{Stability Analysis}

We first recall the following standard result from \cite{BrFo91}.
\begin{thm} If there is a constant such that 
\begin{equation}\label{eq:infsupsmall}
 \boxed{ \| q \|_\Omega + h \| \nabla q \|_\Omega  \lesssim \sup_{\bfv \in W_h} \frac{b(q,\bfv)}{\|\bfv \|_{a_h}} }
\end{equation}
Then there is a constant such that 
\begin{equation}\label{eq:infsupbig}
\boxed{ \tn (\bfv,q )\tn_h \lesssim \sup_{(\bfv,q) \in W_h \times Q_h} \frac{A_h((\bfv,q), (\bfw,r))}{\tn (\bfw,r)\tn_h} }
\end{equation}
\end{thm}

We shall now prove that (\ref{eq:infsupsmall}) holds for the nonconforming space $W_h \times Q_h$ using an approach 
called Verf\"urth's trick \cite{Ver84}, which proceeds in two steps.

\begin{lem}[{\bf Step 1}] There are constants $c_1$ and $c_2$ such that 
\begin{equation}\label{eq:verfurth-a}
\boxed{ \sup_{\bfv \in W_h} \frac{b(\bfv,q)}{\|\bfv \|_{a_h}}  \geq c_1 \| q \| - c_2 h \| \nabla q \| }
\end{equation}
\end{lem}

\begin{proof} For each $q \in Q_h$ there exists a $\bfv\in W$ such that 
\begin{align}
c_1 \| q \|_\Omega \leq \frac{(\nabla \cdot \bfv,q)_\Omega}{\| \nabla \bfv \|_\Omega} 
\end{align}
We shall now replace $\bfv \in W$ by the interpolant $\pi_h \bfv$ and estimate the remainder term as follows
\begin{align}
-\frac{( \pi_h \bfv,\nabla q)_\Omega}{\| \nabla \pi_h \bfv \|_\Omega} 
&=- \frac{(\bfv,  \nabla q)_\Omega}{\| \nabla \pi_h \bfv \|_\Omega}  +  \frac{(\bfv - \pi_h \bfv,  \nabla q)_\Omega}{\| \nabla \pi_h \bfv \|_\Omega}  
\\
&=  \frac{(\nabla \cdot \bfv,  \nabla q)_\Omega}{\| \nabla \bfv \|_\Omega}  \frac{\| \nabla \bfv \|_\Omega}{\| \nabla \pi_h \bfv\|_\Omega} 
+  \frac{(\bfv - \pi_h \bfv,  \nabla q)_\Omega}{\| \nabla \bfv \|_\Omega}  \frac{\| \nabla \bfv \|_\Omega}{\| \nabla \pi_h \bfv\|_\Omega} 
\\
&\geq  c_1 \| q \|_\Omega - 
 \frac{\|\bfv - \pi_h \bfv\|_\Omega \|\nabla q\|_\Omega}{\| \nabla \bfv \|_\Omega}  \frac{\| \nabla \bfv \|_\Omega}{\| \nabla \pi_h \bfv\|_\Omega} 
\\
&\geq c_1 \| q \|_\Omega - c_2 h \| \nabla q \|_\Omega 
\end{align}
Here we used partial integration together with the continuity of $q$ and the boundary condition $\bfv = 0$ on $\partial \Omega$,  the interpolation estimate (\ref{eq:interpol-velocity}) and the boundedness $\| \nabla \pi_h \bfv \|_\Omega \lesssim \| \nabla \bfv \|_\Omega$ of the interpolation operator.  
\end{proof}

\begin{lem}[{\bf Step 2}]\label{lem:step2} Under the assumption (\ref{eq:assump}) there exists $c_3>0$ such that
\begin{equation}\label{eq:step2}
\boxed{ \sup_{\bfv\in W_h}\frac{b(\bfv,q)}{\| \bfv\|_{a_h}}\geq c_3h\,\| \nabla q\| }
\end{equation}
\end{lem}
\begin{proof}  Using partial integration we have the identity 
\begin{align}\label{eq:infsup2-a}
b(\bfv, q) = \sum_{T\in \mcT_h} -(\bfv, \nabla q)_{\mcT_h}  
\end{align}
Observing that $\nabla p$ is element-wise constant we may apply the quadrature formula 
\begin{equation}\label{eq:quadrature}
\int_T w(\bfx) dT \approx \int_T \sum_{E \in \mcE_h(T)}   w(\bfx_{E})\varphi_{T,E}(\bfx_{E}) dT =\frac{\vert T\vert}{6}\sum_{E \in \mcE_h(T)}w(\bfx_{E})
\end{equation}
which is exact for $w \in V_h$,  to obtain 
\begin{align}
\sum_{T\in \mcT_h} -(\bfv, \nabla q)_{\mcT_h} &= \frac{|T|}{6} \sum_{T\in \mcT_h} \sum_{E \in \mcE_h(T)} - \bfv(\bfx_E ) \cdot \nabla p(\bfx_E) 
\end{align}
where $\{\varphi_E\}_{E \in \mcE_h}$ is the global basis in $V_h$,  $\mcE_h(T)$ is the set of edges belong to element $T$,  and $\varphi_{T,E} = \varphi_E|_T$.

Since $q$ is continuous it follows that the tangent derivative $\bft_E \cdot \nabla q$ along each edge $E$, with unit tangent vector $\bft_E$,  
is also continuous and thus taking 
\begin{equation}
\bfv_* = \sum_{E \in \mcE_h^I}  -  h^2 (\bft_E \cdot \nabla q) \bft_E  \varphi_E 
= \sum_{T \in \mcT_h}  \sum_{E \in \mcE^I_h(T)}  - h^2 (\bft_{E} \cdot \nabla q) \bft_{E}  \varphi_{T,E} 
\end{equation}
where $\mcE_h^I \subset \mcE_h$ is the set of interior edges,   we get 
\begin{align}
\sum_{T\in \mcT_h} -(\nabla q, \bfv)_{\mcT_h} &= 
\sum_{T \in \mcT_h} \frac{|T|}{6} \sum_{E \in \mcE_h^I(T)} h^2 (\bft_E \cdot \nabla q (\bfx_E))^2 
\gtrsim 
\sum_{T \in \mcT_h}  h^2 \| \nabla q \|^2_T
\end{align}
since $\nabla q$ is element-wise constant and there are, by assumption,  three linearly independent tangent vectors in the set 
$\{\bft_E \}_{E \in \mcE_h(T)}$ for each element $T\in \mcT_h$.   Finally,  noting that 
\begin{align}
\|\nabla v_*\|_T^2 \lesssim h^4 \sum_{E \in \mcE_h^I(T)} (\bft_E \cdot \nabla q (\bfx_E))^2 \|\nabla \varphi_{T,E}\|_T^2  
\lesssim h^4 \sum_{E \in \mcE^I_h(T)} |\nabla q (\bfx_E)|^2  h^{-2} h^3  
\lesssim h^2 \| \nabla q \|_T^2
\end{align} 
where we used an inverse bound to conclude that $\| \nabla \varphi_E \|_T^2 \lesssim h^{-2} \| \varphi_E \|_T^2 \lesssim h^{-2} h^3$,  and 
summing over $T$ gives
\begin{align}
\| v _* \|_\Omega \lesssim h \| \nabla q \|_\Omega
\end{align}
Combining the estimates we get the desired result since
\begin{align}
\frac{b(v_*,q)}{\| \nabla v_* \|_\Omega} \gtrsim \frac{h^2 \| \nabla q \|^2_\Omega}{\| \nabla v_* \|_\Omega} 
\gtrsim h \| \nabla q \|_\Omega  \frac{h \| \nabla q \|_\Omega}{\| \nabla v_* \|_\Omega} 
\gtrsim  h \| \nabla q \|_\Omega
\end{align}
\end{proof}

\begin{rem} For the alternative  form $\tb$ we get the more complicated expression 
\begin{align}
\tb(\bfv, q) = \sum_{T\in \mcT_h} -(\bfv, \nabla q)_{T}  + (\bfn\cdot \bfv, q)_{\partial T}
\end{align}
where the trace term on the boundary of $T$ does not vanish.  We will however show in  \ref{sec:appendix} that with the same 
choice of $\bfv_*$ the trace term can be shown to be dominated by the bulk term on each element.  The proof is based on mapping to the 
reference element and explicit computation of the two integrals. Thus Lemma \ref{lem:step2} also holds for the form $\tb$.
\end{rem}

\begin{lem} There is a constant such that the inf-sup condition (\ref{eq:infsupsmall}) holds.
\end{lem}
\begin{proof}
Multiply (\ref{eq:verfurth-a}) by $c_3$ and (\ref{eq:step2}) by $2c_2$ and add up to find
\begin{align}
(c_3+2 c_2)\sup_{\bfv \in W_h} \frac{b(q,\bfv)}{\|\bfv \|_{a_h}}  \geq c_3c_1 \| q \|_\Omega - c_3c_2 h \| \nabla q \|_\Omega 
+ 2 c_3c_2h\,\| \nabla q\|_\Omega 
= c_3c_1 \| q \|_\Omega + c_3c_2h\,\| \nabla q\|_\Omega 
\end{align}
Thus (\ref{eq:infsupsmall}) holds with the hidden constant $c_3 = \min(c_1,c_2) /(c_3+2 c_2)$.
\end{proof}

\subsection{Error Analysis}
\begin{thm} There is a constant such that 
\begin{align}
\boxed{ \tn (\bfu - \bfu_h, p - p_h ) \tn_h \lesssim h \Big( \| \bfu \|_{H^3(\Omega)}  + \| p \|_{H^1(\Omega)} \Big) }
\end{align}
\end{thm}
\begin{proof} We first split the error in an interpolation error part  and a discrete part
\begin{align}
\tn (\bfu,p) - (\bfu_h,p_h)\tn_h \leq \tn (\bfu,p) - (\pi_h \bfu,\pi_h p)\tn_h + \tn (\pi_h \bfu,\pi_h p) - ( \bfu_,p_h)\tn_h
\end{align}
The first term can be estimated using the interpolation error estimate (\ref{eq:interpol}). To estimate the discrete part 
of the error we employ the inf-sup condition (\ref{eq:infsupbig}) to obtain
\begin{align}
&\tn (\pi_h \bfu - \bfu_h,  \pi_h p - p_h ) \tn_h 
\lesssim \sup_{(\bfv,q) \in W_h \times Q_h} \frac{ A_h( ( \pi_h \bfu - \bfu_h,  \pi_h p - p_h, (\bfv,q) )  }{ \tn (\bfv,  q) \tn_h}
\\
& \lesssim 
\sup_{(\bfv,q) \in W_h \times Q_h} \frac{ A_h( (\pi_h \bfu - \bfu,  \pi_h p - p ), (\bfv,q) ) }{ \tn (\bfv,  q) \tn_h} 
+ 
\sup_{(\bfv,q) \in W_h \times Q_h} \frac{ A_h( \bfu, p), (\bfv,q) ) - (\bff,\bfv)_\Omega  }{ \tn (\bfv,  q) \tn_h}
\\
& \lesssim \tn (\pi_h \bfu - \bfu,  \pi_h p - p ) \tn_h +
\sup_{(\bfv,q) \in W_h \times Q_h} \frac{ A_h( \bfu, p), (\bfv,q) ) - (\bff,\bfv)_\Omega  }{ \tn (\bfv,  q) \tn_h}
\end{align}
Here we used continuity (\ref{eq:contAh}) of the form $A_h$ for the first term which can now be estimated using the 
interpolation error estimate (\ref{eq:interpol}).  The second term accounts for the consistency error and using partial 
integration we find that
\begin{align}
&A_h( \bfu, p), (\bfv,q) ) - (\bff,\bfv)_\Omega  = a_h(\bfu,\bfv) - b(\bfv,p) + b(\bfu,q)  - (\bff,\bfv)_\Omega
\\
&= \sum_{T \in \mcT_h} (\nabla \bfu, \nabla \bfv)_T + (\bfv,\nabla p )_T - (\bfu, \nabla q)_T - (\bff,\bfv)_T
\\
&= \sum_{T \in \mcT_h} - (\Delta \bfu,  \bfv)_T + ( \nabla_n \bfu, \bfv)_{\partial T} + (\bfv,\nabla p )_T + (\nabla \cdot \bfu, q)_T 
- (\bfn \cdot \bfu, q)_{\partial T} - (\bff,\bfv)_T
\\
&= \sum_{F \in \mcF_h}  (\bfu \otimes \nabla, \bfv\otimes \bfn)_{\partial T} 
\\
&= \sum_{F \in \mcF_h}  ( \bfu \otimes \nabla , [\bfv\otimes \bfn])_{F} 
\end{align}
where $[\bfv \otimes \bfn] = \bfv_1 \otimes \bfn_1 +  \bfv_2\otimes \bfn_2$ for a face $F$ shared by elements 
$T_1$ and $T_2$ with $\bfv|_{T_i} = \bfv_i$, $i=1,2.$ and $[\bfv \otimes \bfn] = \bfv \otimes \bfn$ for a face at the boundary.
Then using the fact that $[\bfv] = \bfzero$ in the midpoints of the edges and that midpoint quadrature is exact for quadratic polynomials 
on  a triangle it follows that $( 1, ([\bfv\otimes \bfn])_{ij})_F =0$ and therefore we may subtract 
the $L^2$-projection on constant functions on the faces and then estimate the contributions using the following standard bounds
\begin{align}\label{eq:apriori-dd}
\sum_{F \in \mcF_h}  (  \bfu\otimes \nabla , [\bfv \otimes \bfn])_{F}  
&= \sum_{F \in \mcF_h}  ((I-P_{0,F}) \bfv \otimes \nabla, (I-P_{0,F}) [\bfv\otimes \bfn ] )_{F} 
\\
&\leq 
\sum_{F \in \mcF_h}  \|(I-P_{0,F}) \bfu \otimes \nabla \|_F \|(I-P_{0,F}) [\bfv\otimes \bfn] \|_F
\\
&\label{eq:apriori-ee}
  \lesssim
\sum_{F \in \mcF_h}  h \| \nabla_F  (\bfu \otimes \nabla ) \|_F h \| \nabla_F [ \bfv\otimes \bfn ]\|_F
\\
& \lesssim
\sum_{T \in \mcT_h}  h \| \bfu\|_{H^3(T)}   \| \nabla \bfv \|_T
\end{align}
where $\nabla_F = (I - \bfn_F \otimes \bfn) \nabla$,  with $\bfn_F$ a unit normal to the $F$,  is the tangential gradient to the face $F$.
Combing the bounds gives the desired estimate.
\end{proof}

\begin{rem} For the form $\tb$ we get a consistency error  of the form
\begin{align}
&A_h( \bfu, p), (\bfv,q) ) - (\bff,\bfv)_\Omega  = a_h(\bfu,\bfv) - b(\bfv,p) + b(\bfu,q)  - (\bff,\bfv)_\Omega
\\
&= \sum_{T \in \mcT_h} (\nabla \bfu, \nabla \bfv)_T - (\nabla \cdot \bfv, p )_T + (\nabla \cdot \bfu,  q)_T - (\bff,\bfv)_T
\\
&= \sum_{T \in \mcT_h} - (\Delta \bfu,  \bfv)_T + ( \nabla_n \bfu, \bfv)_{\partial T} 
+ (\bfv,\nabla p )_T  - (\bfn \cdot \bfv, p)_{\partial T} - (\bff,\bfv)_T+ (\nabla \cdot \bfu, q)_T
\\
&= \sum_{F \in \mcF_h}  ( \bfu\otimes \nabla , \bfv\otimes \bfn)_{\partial T}   - (\bfn \cdot \bfv, p)_{\partial T}
\\
&= \sum_{F \in \mcF_h}  ( \bfu\otimes \nabla ,  [\bfv\otimes \bfn ])_{F}  - ([\bfn \cdot \bfv], p)_{\partial T}
\end{align}

which,  using the same approach as in estimates (\ref{eq:apriori-dd}-\ref{eq:apriori-ee}),  can be estimated by
\begin{align}
 \sum_{F \in \mcF_h}  ( \bfu\otimes \nabla , [\bfv\otimes \bfn])_{F}  - ([\bfn \cdot \bfv], p)_{\partial T} 
& \lesssim \sum_{T \in \mcT_h}  h \| \bfu\|_{H^3(T)}   \| \nabla \bfv \|_T + h \| \bfv \|_{H^1(T)} \| p \|_{H^1(T)}
 \\
 & \lesssim \sum_{T \in \mcT_h} h (  \| \bfu\|_{H^3(T)}  + \| p \|_{H^1(T)} )  \| \bfv \|_{H^1(T)} 
\end{align}
\end{rem} 

%
%
%
%

\section{NUMERICAL EXAMPLES}\label{numex}

\subsection{Convergence}

We consider a problem in the ball with radius 1 and with center at the origin. A fabricated solution is given by
\begin{equation}
\bfu = (x_2^3 - x_3^3,x_1^3 - x_3^3,-x_1^3 - x_2^3), \quad p = 6 (x_1 x_2 - x_1x_3 -x_2x_3)
\end{equation}
with $\bff={\bf 0}$.
The exact solution is used as Dirichlet data and zero mean pressure is imposed by a Lagrange multiplier. We compare the convergence of the pressure inconsistent method, using $\tilde{b}(\bfu,q)$, to
the pressure consistent method, using $b(\bfu,q)$, in Fig. \ref{fig:convergence}. The convergence is shown in $L_2$ for the pressure and the velocity and in broken $H^1$ semi--norm for the velocity.
We note that the methods converge at the same rate, albeit with a slightly larger error constant for the inconsistent method.
The observed rates from Fig. \ref{fig:convergence} are 
\begin{align}
\Vert\bfu-\bfu_h\Vert_{L_2(\Omega)} = {}& O(h^2)\\
\Vert\bfu-\bfu_h\Vert_{a_h} ={}& O(h)\\
\Vert p-p_h\Vert_{L_2(\Omega)} \approx {}&  O(h^{3/2})
\end{align} 
As we do not have full $P^2$ approximation for the velocity, the observed rate of convergence of pressure is better than expected. 

In Fig. \ref{fig:pressure} we compare the pressure solutions to the interpolated pressure on the boundary, shown on the finest mesh in the sequence used for the convergence study.

\subsection{Laplacian Form vs. Strain Form}\label{sec:numex2}

In this example we show that the strain formulation of Stokes equations poses no problem as we have a Korn inequality for our approximation. We consider a Poiseuille type problem in the domain $\Omega = (0,3)\times (0,1)\times (0,1/10)$
with boundary conditions $\bfu = (x_2(1-x_2),0,0)$ at $x_1=0$, at $x_2=0$, and at $x_2=1$, and with $u_3=0$ at $x_3=0$ and at $x_3=1/10$. 
In Fig. \ref{fig:velo} we show the velocity field in the $(x_1,x_2)$--plane and we note that the strain formulation gives a stress free condition at the outflow. In Fig. \ref{fig:pre} we show the corresponding pressure. The computations were made using the $\tilde{b}$--form of the side condition.

\begin{figure}
\begin{center}
\includegraphics[scale=0.25]{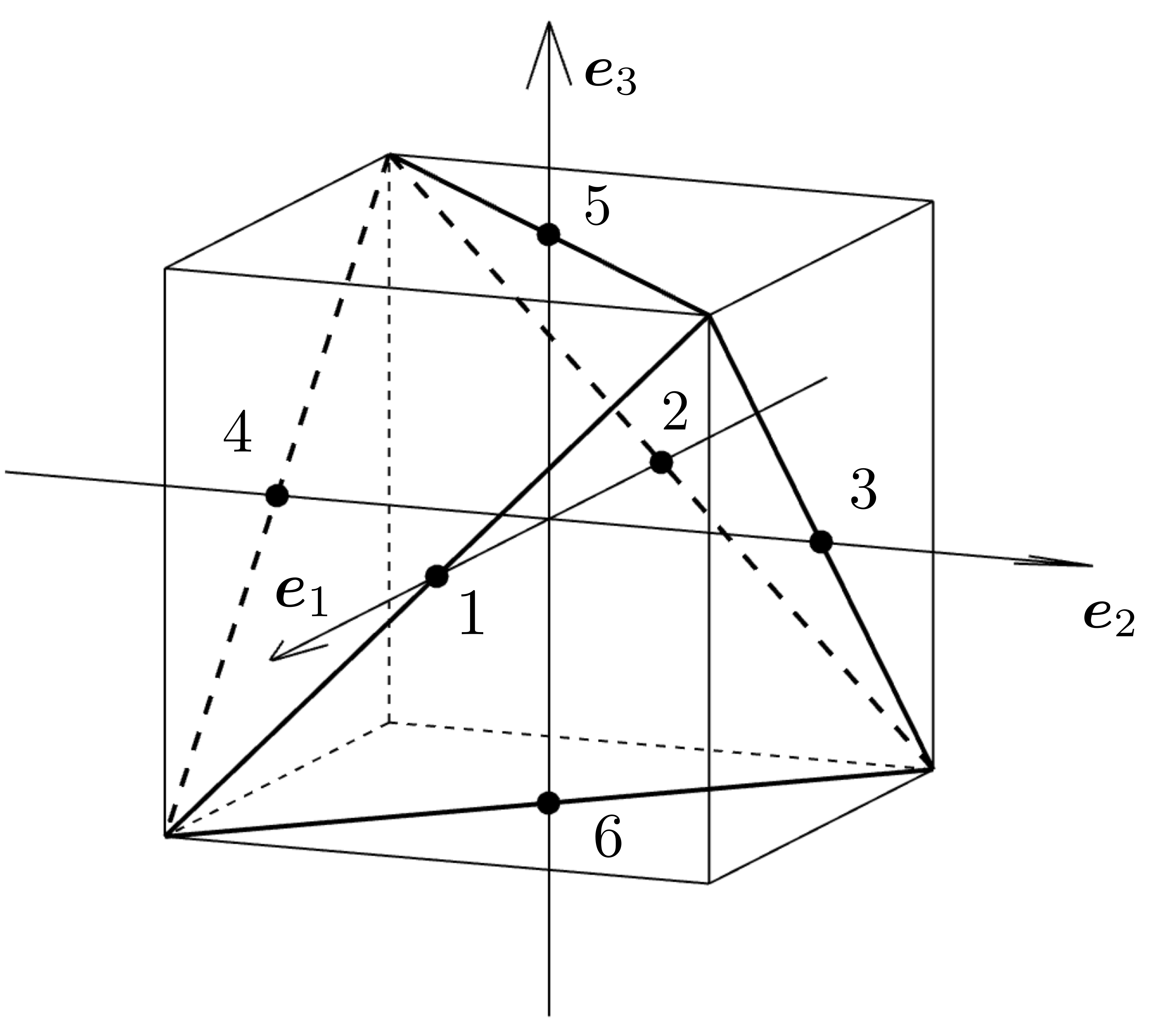}
\end{center}
\caption{The reference element and the enumeration of the six degrees of freedom.}
\label{fig:refelement}
\end{figure}

\begin{figure}
\begin{center}
\includegraphics[scale=0.25]{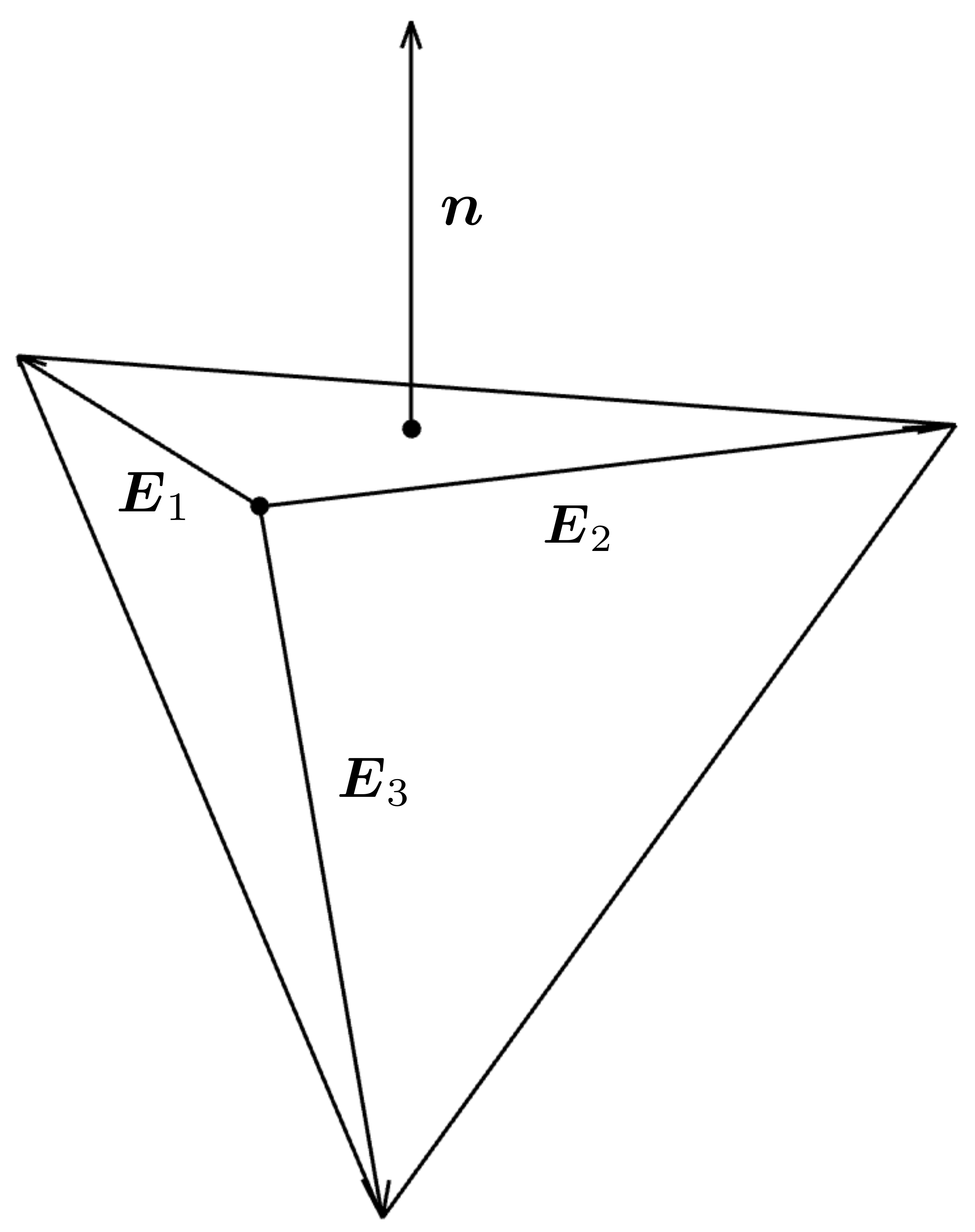}
\end{center}
\caption{The normal $\bfn$ and the three edge vectors $\{\bfE_1,\bfE_2,\bfE_3\}$.}
\label{fig:infsup}
\end{figure}

\begin{figure}
\begin{center}
\includegraphics[scale=0.25]{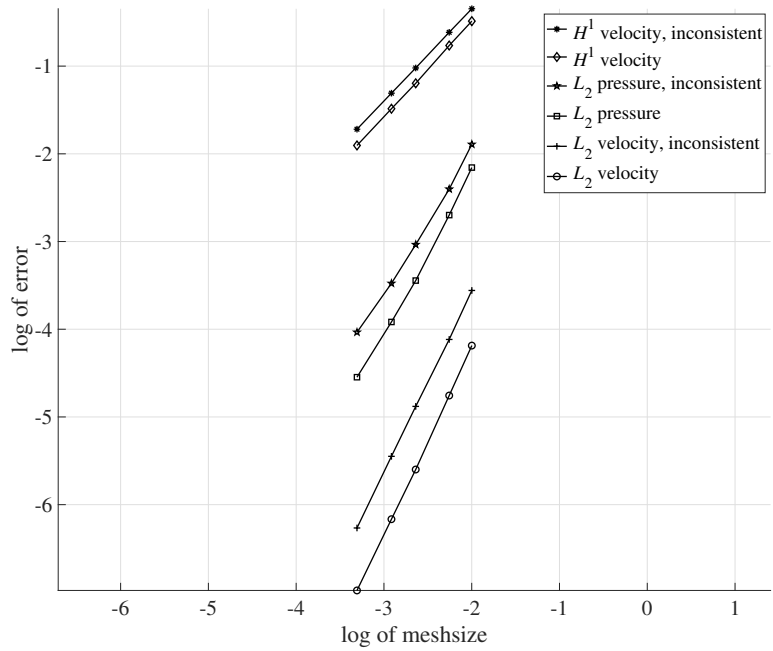}
\end{center}
\caption{Convergence of velocity and pressure for the pressure inconsistent $\tilde{b}$--form and the consistent $b$--form.}
\label{fig:convergence}
\end{figure}
\begin{figure}
\begin{center}
\includegraphics[scale=0.12]{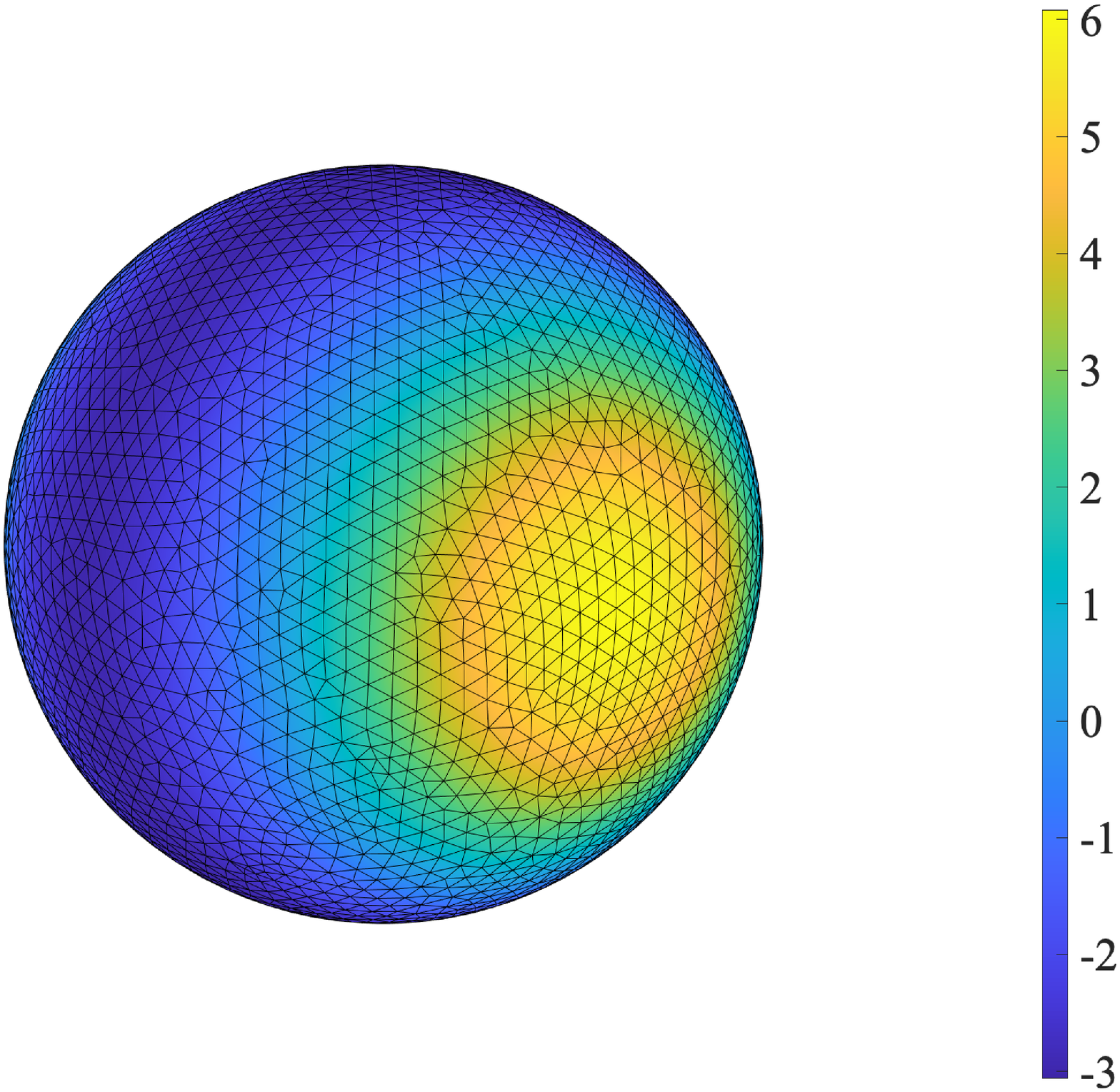}\includegraphics[scale=0.12]{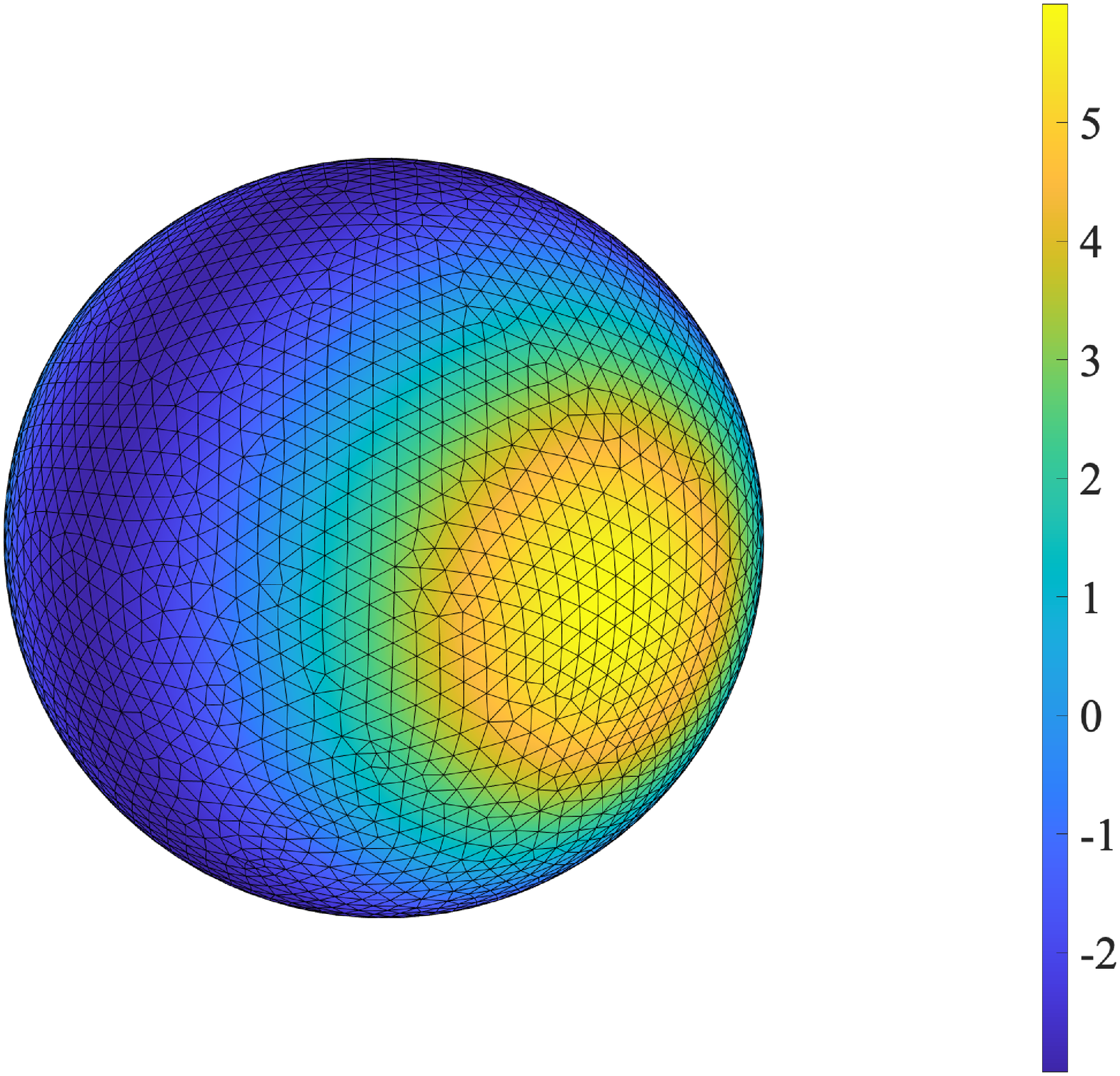}\includegraphics[scale=0.12]{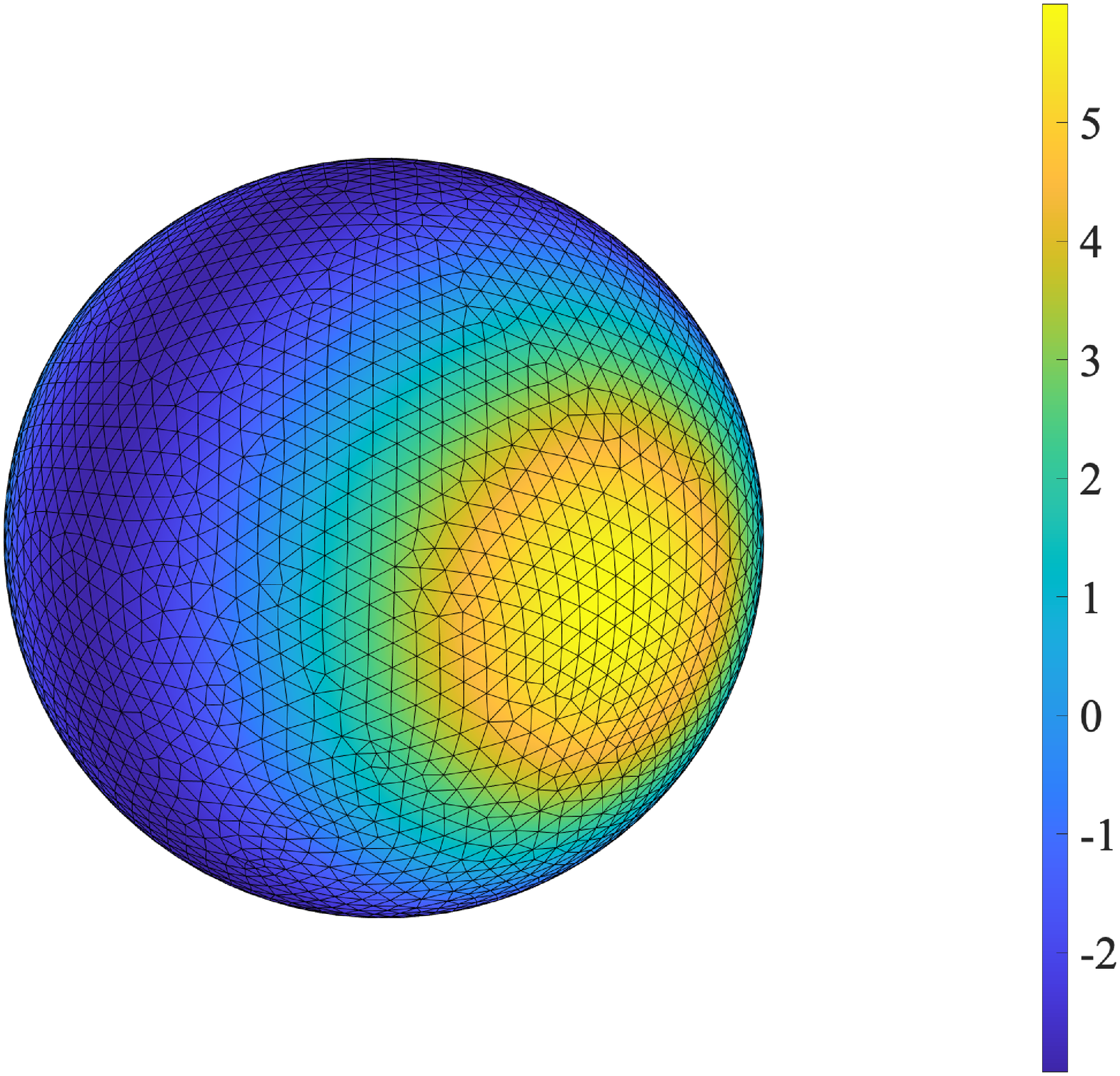}
\end{center}
\caption{Pressure solutions: inconsistent (left), consistent (center) interpolated (right).}
\label{fig:pressure}
\end{figure}
\begin{figure}
\begin{center}
\includegraphics[scale=0.2]{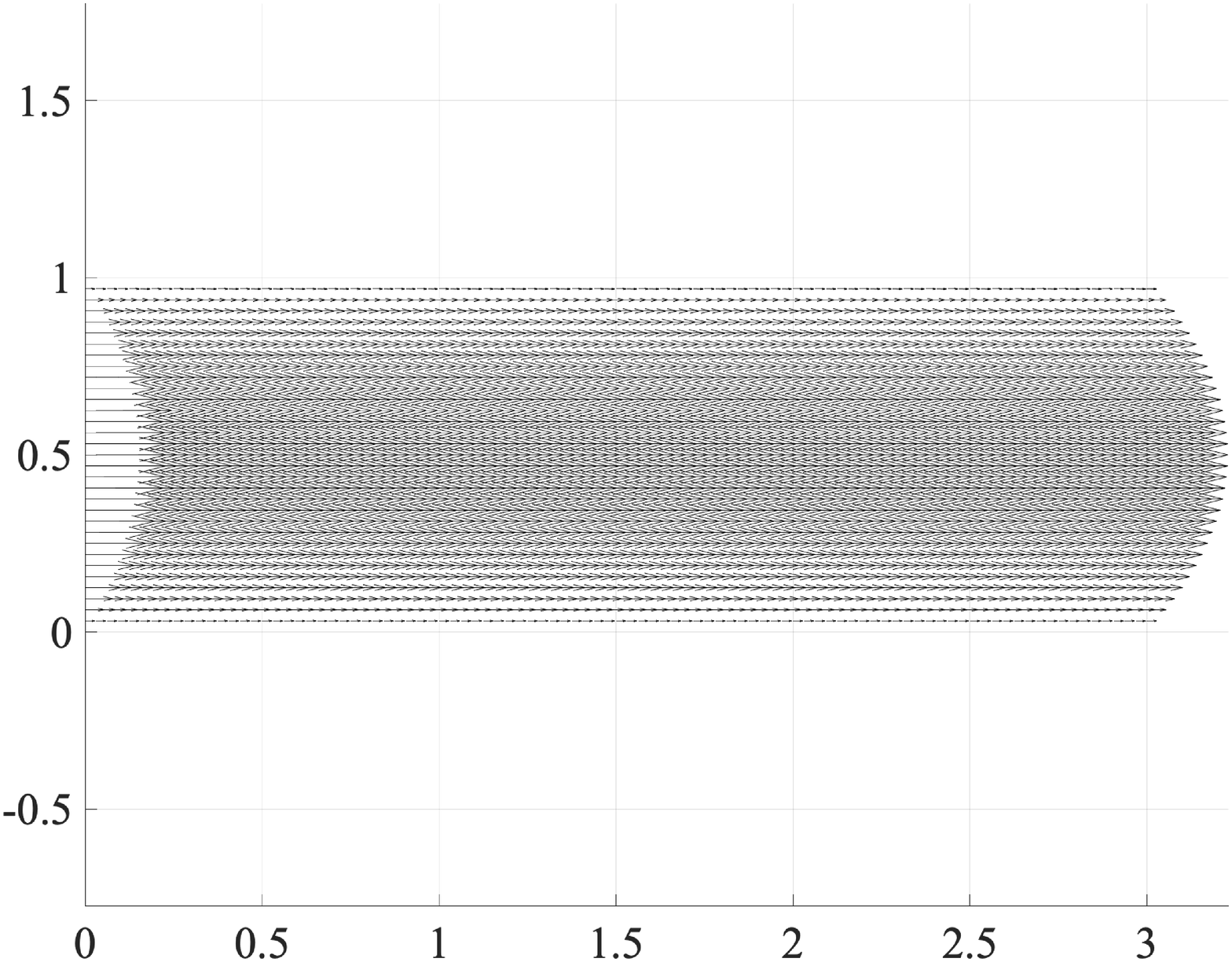}\includegraphics[scale=0.2]{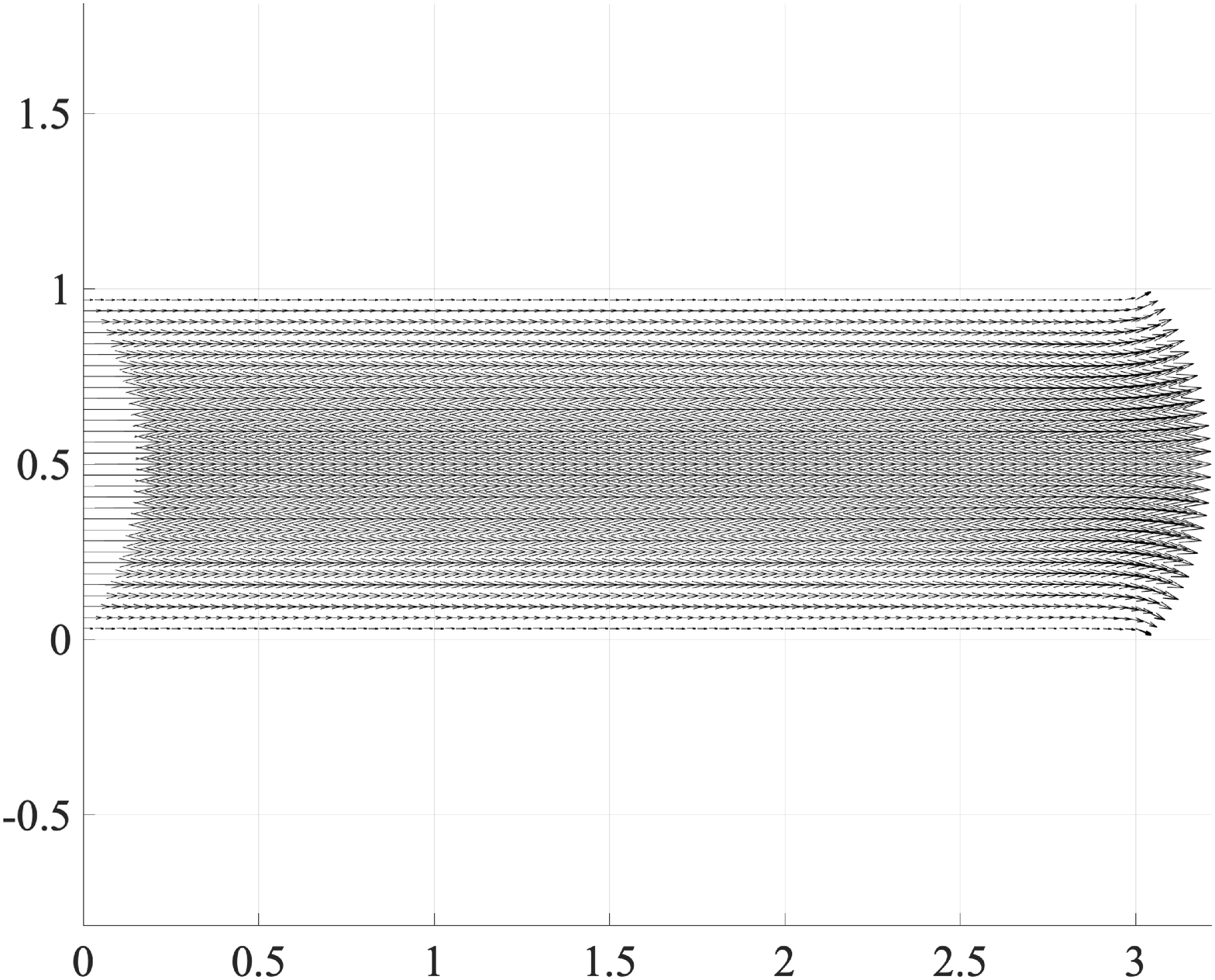}\end{center}
\caption{Velocity for the Laplacian form (left) and strain form (right).}
\label{fig:velo}
\begin{center}
\includegraphics[scale=0.2]{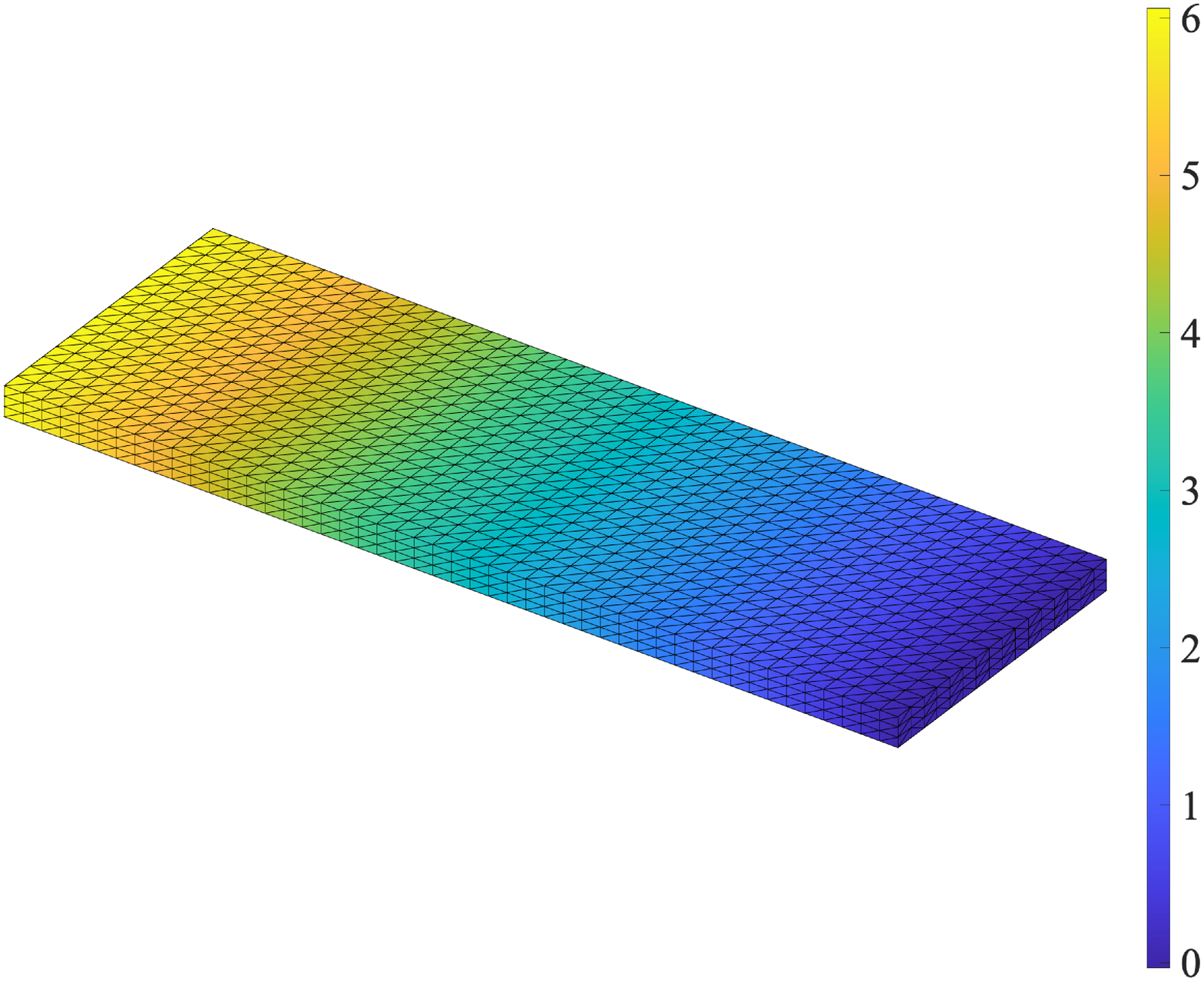}\includegraphics[scale=0.2]{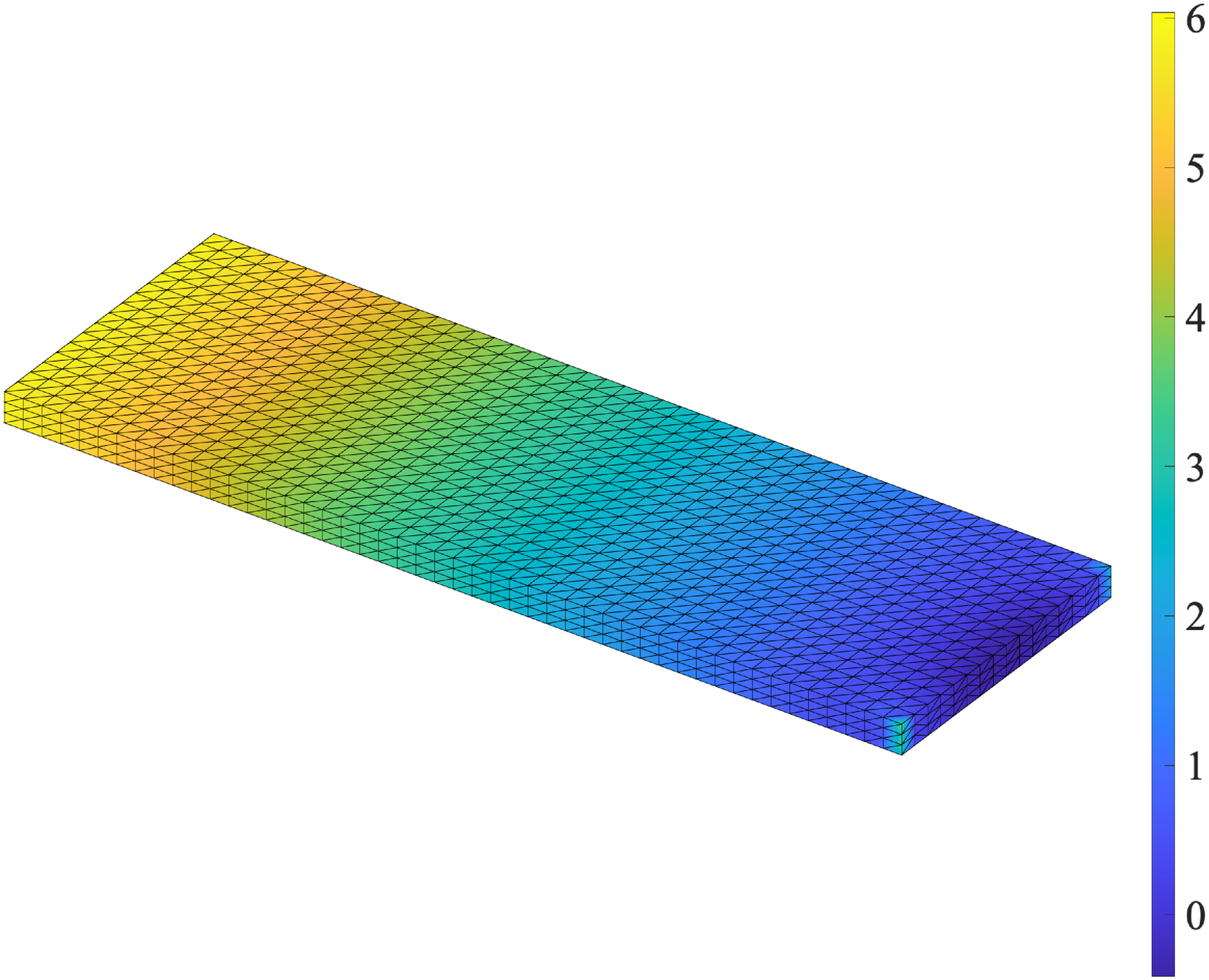}\end{center}
\caption{Pressure for the Laplacian form (left) and strain form (right).}
\label{fig:pre}
\end{figure}
\appendix

\section{INF-SUP CONDITION FOR THE $\boldmath{\tb}$ FORM}
\label{sec:appendix}

We shall now verify that the the inf-sup condition (\ref{eq:infsupsmall}) holds also for the form $\tb$.  The element contribution for $\tb$ is 
\begin{align}
\tb_T(\bfv,q) &= (\nabla \cdot \bfv, q)_T 
=  - (\bfv,\nabla  q)_T + (\bfn \cdot \bfv, q )_{\partial T}
\end{align}
and as in the proof of Lemma \ref{lem:step2} we take
\begin{equation}
\bfv_* = - \sum_{E \in \mcE_h^I(T)} (\bft_E \cdot \nabla q(\bfx_E) ) \bft_E  \varphi_E 
\end{equation}
which gives 
\begin{align}
\tb_T(\bfv,q) 
&=  \sum_{E \in \mcE_h^I(T)}  (\bft_E \cdot \nabla  q(\bfx_E) \varphi_E ,\bft_E \cdot \nabla  q)_T - (\bfn \cdot \bft_E (\bft_E \cdot \nabla q (\bfx_E))  \varphi_E, q - P_{0,T} q )_{\partial T}
\end{align}
where $P_{0,T}$ is the $L^2$-projection on constants on $T$.  Here we used the fact that 
\begin{equation}\label{eq:app-a}
(\bfn \cdot \bft_E \bft_E \cdot \nabla q (\bfx_E))  \varphi_E,  1 )_{\partial T} = 0
\end{equation}
to subtract $P_{0,T} q$.  Identity (\ref{eq:app-a}) holds since if $F$ is a face that has $E$ as one of its edges then $\bfn \cdot \bft_E = 0$ 
and if $F$ does not have $E$ as an edge then $\varphi_E$ is zero in the midpoints,  which imply $(\varphi_E,1)_F=0$ since midpoint 
quadrature is exact for quadratic polynomials on triangles.

Next using the fact that  $q$ is linear on $T$ we may write
\begin{equation}
q = a + \bfb\cdot(\bfx - \bfx_T)
\end{equation}
where $\bfx_T$ is the center of gravity of $T$.  We then have $(I - P_{0,T}) q = \bfb\cdot(\bfx - \bfx_T)$ and 
\begin{align}
\tb_T(\bfv_*,q) 
&=  \sum_{E \in \mcE_h^I(T)}  (\bft_E \cdot \bfb  \varphi_E(\bfx) ,\bft_E \cdot \bfb )_T - (\bfn \cdot \bft_E (\bft_E \cdot \bfb)  \varphi_E(\bfx), \bfb\cdot (\bfx - \bfx_T) )_{\partial T}
\end{align}
We now study the contributions for each $E \in \mcE_h^I$. The bulk term may be directly computed using (\ref{eq:quadrature}) as
\begin{equation}
 (\bft_E \cdot \bfb  \varphi_E (\bfx) ,\bft_E \cdot \bfb )_T =  (\bft_E \cdot \bfb)^2 \int_T \varphi_E(\bfx) =  (\bft_E \cdot \bfb)^2  |T|/6
\end{equation}
The formula holds for general elements since the Jacobian of the 
affine mapping is constant.  It remains to study  the boundary contribution 
\begin{align}
g(\bfb) = -  \bft_E \cdot \bfb \int_{\partial T}  \bfn \cdot \bft_E  \varphi_E(\bfx) \bfb\cdot (\bfx - \bfx_T)
\end{align}
which is a quadratic function in $\bfb$.

\paragraph{Computation on the Reference Element}

Let us first look at the reference element $\hat T$ as defined in Section \ref{q1app}.  Expanding $\bfb$ in the orthonormal basis $\{\bfe_1, \bft_E, \bfs_E\}$,  we obtain by symmetry 
\begin{align}
 \int_{\partial T}  \bfn \cdot \bft_E  \varphi_E \bfb \cdot (\bfx - \bfx_T)
= 
\begin{cases}
0 & \bfb = \bfe_1
\\
0 &\bfb = \bfs_E
\end{cases}
\end{align}
Therefore,  
\begin{align}
g(\bfb) = - (\bft_E\cdot \bfb)^2  \int_{\partial T}  \bfn \cdot \bft_E  \varphi_E(\bfx) \bft_E \cdot (\bfx - \bfx_T) 
=  - 2 (\bft_E\cdot \bfb)^2  \underbrace{ \int_{F^*(E)}  \bfn \cdot \bft_E  \varphi_E(\bfx) \bft_E \cdot (\bfx - \bfx_T) }_{I}
\end{align}
where $F^*(E)$ is one of the two faces faces to which $E$ does not belong and we again used symmetry to conclude 
that the contributions from the two faces are identical.  We thus obtain
\begin{equation}
b(\bfv_*,q) = \sum_{E \in \mcE_h(T)} (\bft_E \cdot \bfb)^2 ( | T|/6 -   2 I )
\end{equation}
Here $|T|/6 = (8/3)/6 = 4/9$ and direct computation gives $2I= 4/15$ and thus
\begin{align}\label{eq:app-b}
\boxed{ b(\bfv_*,q) = \sum_{E \in \mcE_h^I(T)} (\bft_E \cdot \bfb)^2 \left( \frac{4}{9} - \frac{4}{15}\right) = \frac{8}{45}  \sum_{E \in \mcE_h^I(T)} (\bft_E \cdot \bfb)^2}
\end{align}

\paragraph{Mapped Element} Next consider a mapped element 
\begin{align}
F: {\hat T} \rightarrow T
\end{align}
where $F$ is an affine map of the form
\begin{equation}
\bfx = F(\hat \bfx) = A \hat \bfx + \bfx_T 
\end{equation} 
which means that 
\begin{equation}
\bfx -  \bfx_T= F(\hat \bfx) = A \hat \bfx 
\end{equation} 
where we use the usual hat notation for quantities on the reference element.  Tangent vectors are mapped to tangent vectors
\begin{equation}
A \hat\bft_E = \| A \hat\bft_E \| \bft_E 
\end{equation}
where 
\begin{equation}
\| A \hat\bft_E \| = h_E / h_{\hat E}
\end{equation}
and given $\bfb$ associated with $T$ we define 
\begin{equation}
\hat \bfb_E = A^T \bfb / \| A \hat{\bft}_E \|
\end{equation}
which leads to the identity 
\begin{equation}
 \hat \bft_E \cdot \hat \bfb_E =  \bft_E \cdot \bfb   
\end{equation}
As above the element contribution takes the form 
\begin{align}
 \int_T \bft_E \cdot \bfb  \varphi_E(\bfx) \bft_E \cdot \bfb  
 &= (\bft_E \cdot \bfb)^2 |T|/6  = |T | \, | \hat T|^{-1}  (\hat \bft_E \cdot \hat \bfb_E)^2 |\hat T|/6
\end{align}
Next for the boundary contribution
\begin{align}
&\bft_E \cdot \bfb \int_F \bfn \cdot \bft_E  \varphi_E(\bfx) \bfb\cdot (\bfx - \bfx_T)  dF
\\ 
&\qquad  =
 \hat \bft_E \cdot\hat \bfb \int_{\hat{F}} \hat{\varphi}_{\hat{E}}(\hat \bfx ) (A^T \bfb ) \cdot \hat{\bfx}\, \bfn \cdot \bft_E  |A|_F d\hat F
\\
&\qquad  =
 |T| \, |\hat T|^{-1}  \hat \bft_E \cdot \hat \bfb
 \int_{\hat{F}} \hat \bfn \cdot \hat \bft_E \, \hat{\varphi}_{\hat{E}}(\hat \bfx )\, \| A \hat \bft_{\hat E} \|^{-1} A^T \bfb  \cdot \hat {\bfx} d\hat F
 \\
&\qquad  =
 |T| \, |\hat T|^{-1}  \hat \bft_E \cdot \hat \bfb
 \int_{\hat{F}} \hat \bfn \cdot \hat \bft_E \,\hat{\varphi}_{\hat{E}}(\hat \bfx )\,  \hat \bfb  \cdot \hat {\bfx} d\hat F
\end{align}
Here $|A|_F$ is the Jacobian associated with the mapping $A:\hat F \rightarrow F$ and we used the identity 
\begin{align}\label{eq:facecaj}
\bfn \cdot \bft_E  |A|_F = |T| \, |\hat T|^{-1} \| A \hat \bft_{\hat E} \|^{-1} \hat \bfn \cdot \hat \bft_{\hat E}
\end{align}
which we verify below.   We conclude that after transformation back to the reference element we get 
an expression that has the same form as in the reference element case and thus  we may apply 
(\ref{eq:app-b}) to get 
\begin{align}
b_T(\bfv_*, q) = |T| \, |\hat T|^{-1}  \frac{8}{45}  \sum_{E \in \mcE_h^I} (\hat \bft_E\cdot \hat \bfb_E )^2 
= |T| \, |\hat T|^{-1}  \frac{8}{45}  \sum_{E \in \mcE_h^I} ( \bft_E\cdot \bfb )^2 
= |\hat T|^{-1}  \frac{8}{45}  \sum_{E \in \mcE_h^I} \| \bft_E\cdot \nabla q \|_T^2
\end{align}
Summing over $T \in \mcT_h$ gives
\begin{align}
\boxed{ b(\bfv_*, q) = \sum_{T \in \mcT_h} b_T(\bfv_*, q) =|\hat T|^{-1}  \frac{8}{45}  \sum_{T \in \mcT_h}  \sum_{E \in \mcE_h^I} \| \bft_E\cdot \nabla q \|_T^2 \gtrsim \sum_{T \in \mcT_h} \| \nabla q \|^2_T }
\end{align}
 and we note that the deformation of the element only has an effect on the last inequality.
 
 \paragraph{Verification of (\ref{eq:facecaj})}
 Let $\{\bfE_1,\bfE_2,\bfE_3\}$ be edge vectors to $T$ such that, by well known relations with $\det(\bfa, \bfb, \bfc)$ as the determinant of the matrix whose columns are $\bfa$, $\bfb$, and $\bfc$, 
 \begin{equation}
6 |T| = \det(\bfE_1, \bfE_2, \bfE_3)
 \end{equation}
and 
\begin{equation}
2|F|= \det(\bfn, \bfE_2,\bfE_1)
\end{equation}
see Fig.  \ref{fig:infsup}.  
We then have 
\begin{align}
6 |T| &= \det(\bfE_1,\bfE_2,\bfE_3) = \det(\bfE_1,\bfE_2,(\bfE_3\cdot \bfn) \bfn ) = \det(\bfE_1,\bfE_2,(\bfE_3\cdot \bfn) \bfn )
\\
&\qquad 
= (\bfE_3\cdot \bfn) \det(\bfE_1,\bfE_2, \bfn ) 
= -(\bfE_3\cdot \bfn) \det(\bfn,\bfE_2,  \bfE_1) = -(\bfE_3\cdot \bfn) 2 |F|
\end{align}
Therefore 
\begin{align}
-(\bfE_3\cdot \bfn) |F| = 3|T| = |T| \,  | {\hat{T}}|^{-1}  3 | {\hat{T}}| = - |T| \,  | {\hat{T}}|^{-1} (\hat \bfE_3\cdot \hat \bfn) |\hat F|
\end{align}
or 
\begin{align}
h_E (\bft_3\cdot \bfn) |F| = |T| \,  | {\hat{T}}|^{-1} \hat{h}_E (\hat \bft_3\cdot \hat \bfn) |\hat F|
\end{align}
which finally imply 
\begin{align}
 (\bft_3\cdot \bfn) |F|  = |T| \,  | {\hat{T}}|^{-1} \hat{h}_E h_E^{-1} (\hat \bft_3\cdot \hat \bfn) |\hat F| 
 = |T| \,  | {\hat{T}}|^{-1} \| A \hat{\bft}_3 \|^{-1} (\hat \bft_3\cdot \hat \bfn) |\hat F|
\end{align}
\paragraph{\bf Acknowledgement} This research was supported in part by the Swedish Research
Council Grants Nos.\  2017-03911, 2018-05262,  2021-04925,  and the Swedish
Research Programme Essence.

\end{document}